\documentclass{amsart}
\usepackage{amsfonts}
\usepackage{amssymb}
\usepackage{amsmath}
\usepackage{newlfont}
\usepackage[dvips]{graphicx}
\usepackage[mathcal]{euscript}

\newtheorem{thm}{Theorem}
\newtheorem{dfn}[thm]{Definition}
\newtheorem{lem}[thm]{Lemma}

\newtheorem{prop}{Proposition}

\def\R{\mathbb{R}}
\def\Z{\mathbb{Z}}
\def\C{\mathbb{C}}

\def\N{\mathbb{N}}
\def\c{\mathcal{C}}
\def\s{\sigma}

\def\re{\mathfrak{Re}}
\def\im{\mathfrak{Im}}

\def\f{\varphi}
\def\z{\overline{z}}

\def\e{\varepsilon}

\def\cl{\operatorname{cl}}

\def\Arg{\operatorname{Arg}}

\def\id{\operatorname{id}}

 \newcommand{\set}[1]{\left\{#1\right\}}

\newcommand{\ra}{\rightarrow}

\DeclareMathOperator{\Per}{Per}

\DeclareMathOperator{\diam}{diam}

 \newcommand{\A}{\mathcal{A}}

 \newcommand{\orb}{\text{Orb}}

\begin{document}


\title{Complicated dynamics in planar polynomial equations}

\author{Pawe\l \mbox{ W}ilczy\'nski}
\address{Faculty of Mathematics and Information Science, Warsaw University of Technology,
ul. Koszykowa 75  00-662 Warszawa, Poland}
\email{pawel.wilczynski@mini.pw.edu.pl}

\begin{abstract}
We deal with a mechanism of generating distributional chaos in planar nonautonomous ODEs and try to measure chaosity in terms of topological entropy. It is based on the interplay between simple periodic solutions. We prove the existence of infinitely many heteroclinic solutions betwen the periodic ones.
\end{abstract}

\maketitle

\section{Introduction}

The goal of the paper is to investigate the mechanism of generating chaos in the equation
\begin{equation}
\label{egh:og}
\dot{z}=P_n(t,\z) + f(t,z),
\end{equation}
where $P_n$ is a polynomial in $\z$ variable of degree $n$ and it is continuous and periodic in $t$ variable. $f$ is treated as a sufficiently small perturbation.

We deal with $n=2$ and investigate the model equation
\begin{equation}
\label{egh}
\dot{z}=v(t,z)=Re^{i t}(\z^2-1) + f(t,z),
\end{equation}
for degree $2$.

Investigated equations are different from the ones described in \cite{SrzedWoj1, Pien1}. We deal with distributional chaos which is not equivalent to the notion of chaos from \cite{SrzedWoj1, Pien1} but in some cases (see \cite{OpWil3}) may be implied by it. We also try to set a lower boundary for the topological entropy for the equations.

The main result of the paper in the case of big leading coefficient is the following theorem.

\begin{thm}
\label{thm:glowne}
Let $f\in \c(\R\times\C,\C)$ be $2\pi$-periodic in the first variable i.e. $f(t,z)=f(t+2\pi,z)$ for every $(t,z)\in \R\times \C$. Moreover, let
\begin{equation}
\label{ineq:R}
R\geq 1
\end{equation}
and
\begin{align}
\label{ineq:RN}
N\leq & 0.001R
\end{align}
where
\begin{align}
\label{ineq:N}
|f(t,z)|&\leq N,\\
\label{ineq:lipschitz}
|f(t,z) - f(t,w)| & \leq N |z-w|
\end{align}
be satisfied for every $t\in \R$ and $z, w\in Q=\set{p\in \C: |p|< 3}$.

Then 
\begin{enumerate}
\item
the equation \eqref{egh} is distributionally chaotic, 
\item both trivial solutions $-1$ and $1$ for the case $f\equiv 0$ continue to $2\pi$-periodic ones and there exist infinitely many heteroclinic solutions between them,
\item there exists $\beta>0$ and $\Lambda\subset Q$ invariant with respect to the Poincar\'{e} operator $\f_{(-\beta, 2\pi)}$ such that the dynamical system $(\Lambda, \f_{(-\beta, 2\pi)}\mid _\Lambda)$ is semiconjugated to the dynamical system $(\Sigma, \sigma^2)$, where $\Sigma$ is the sofic shift given by the Figure \ref{fig:shift} and $\sigma: \Sigma \longrightarrow \Sigma$ is the shift operator.
\begin{figure}[!htbp]
\begin{center}
\includegraphics[width=100pt,height=300pt,angle=270]{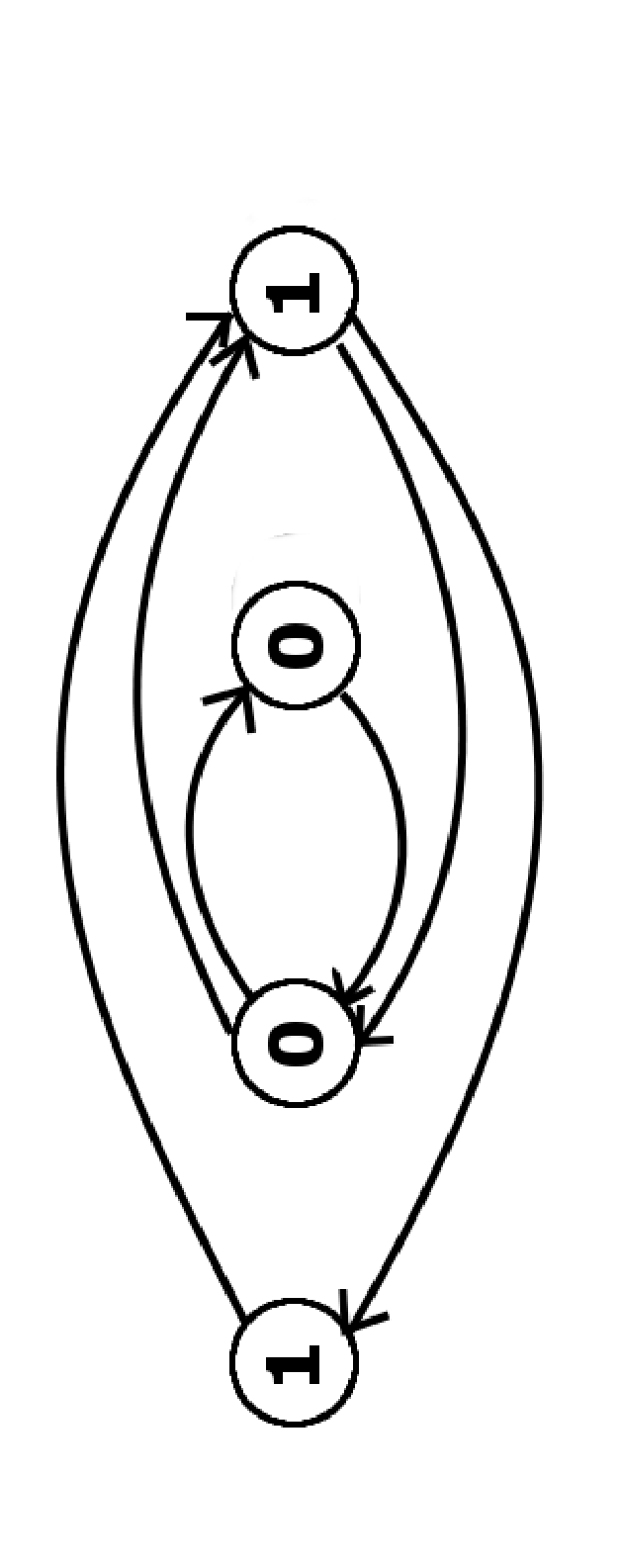}
\end{center}
\caption{Vertex graph of the sofic shift $\Sigma$ for the equation \eqref{egh} with big leading coefficient.}
\label{fig:shift}
\end{figure}
\end{enumerate}
\end{thm}

Since the topological entropy of the system $(\Sigma, \sigma^2)$ is given by $\ln \frac{3+\sqrt{5}}{2}\approx \ln(2.615)$ Theorem \ref{thm:glowne} gives greater entropy than $\ln(2)$ which could be obtained by semiconjugacy to the $(\Sigma_2, \sigma)$. Although we now get higher lower estimation of topological entropy in equation \eqref{egh} we do not know whether new semiconjugacy is now conjugacy.

The main result of the paper in the case of small leading coefficient is the following theorem.

\begin{thm}
\label{thm:male2}
There exist $L>0$ and $\varepsilon_0>0$ such that the equation
\begin{equation}
\dot{z}=v_1(t,z)= L\left[-iz+\z^2-\varepsilon e^{it}\right]
\label{egh:male2}
\end{equation}
is distributionally chaotic provided that
\begin{equation}
\label{ineq:eps}
0< \varepsilon\leq \varepsilon_0.
\end{equation}
Moreover, there exists
$\beta>0$ and $\Lambda\subset \C$ invariant with respect to the Poincar\'{e} operator $\f_{(-\beta, 2\pi)}$ such that the dynamical system $(\Lambda, \f_{(-\beta, 2\pi)}\mid _\Lambda)$ is semiconjugated to the dynamical system $\left(\hat{\Sigma}, \sigma\right)$, where $\hat{\Sigma}$ is the sofic shift given by the Figure \ref{fig:shift2} and $\sigma: \hat{\Sigma} \longrightarrow \hat{\Sigma}$ is the shift operator.
\begin{figure}[!htbp]
\begin{center}
\includegraphics[width=200pt,height=150pt,angle=0]{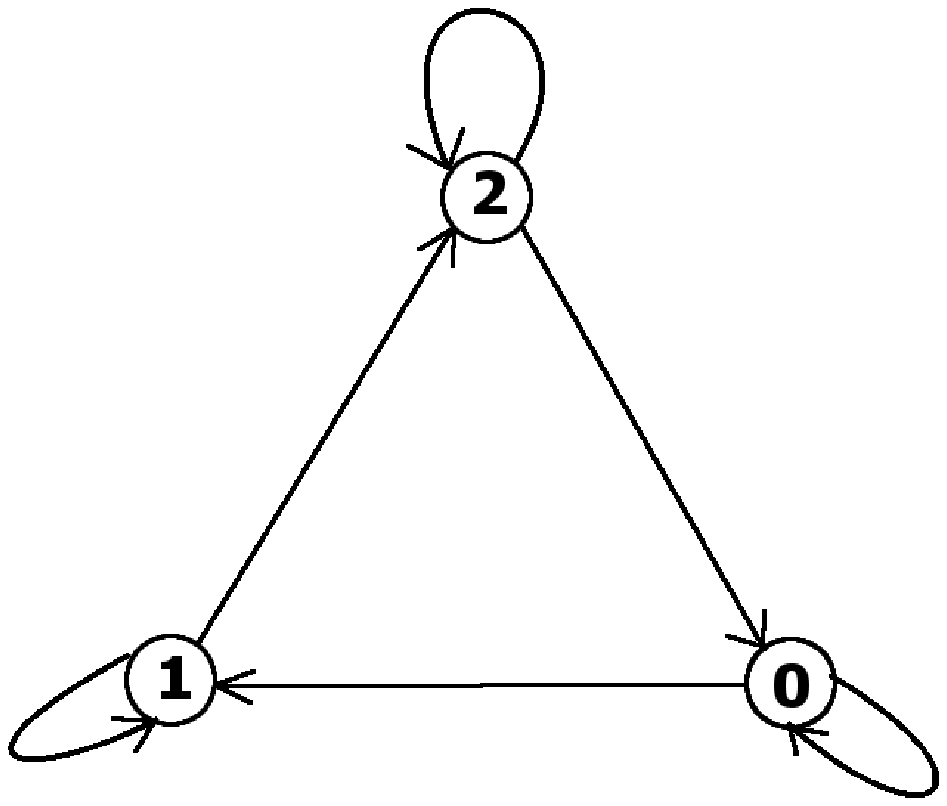}
\end{center}
\caption{Vertex graph of the sofic shift $\hat{\Sigma}$ for the equation \eqref{egh:male2} with small leading coefficient.}
\label{fig:shift2}
\end{figure}

\end{thm}

The change of variables 
\begin{equation}
\label{eq:zmiana}
u(t)=\frac{1}{\varepsilon}e^{\frac{i}{3}t}z\left( \frac{2}{3}t\right)
\end{equation}
in the equation \eqref{egh:male2} gives
\begin{equation}
\dot{u}=v_2(t,u)=\frac{2}{3}L\varepsilon e^{i t}(\overline{u}^2-1) + \left(\frac{1}{2}- L\right) iu.
\label{eq:glowne:m2}
\end{equation}

Topological entropy of the system $\left(\hat{\Sigma}, \sigma\right)$ equals $\ln 2$. But since the time scaling in \eqref{eq:zmiana}, topological entropy for discrete system generated by \eqref{eq:glowne:m2} equals $\frac{2}{3}\ln 2$, so it is lower than in the case of $(\Sigma, \sigma^2)$.

The most interesting case seems to be $L=\frac{1}{2}$ when the equation \eqref{eq:glowne:m2} reduces to 
\begin{equation*}
\dot{u}=v_2(t,u)=\frac{\varepsilon}{3} e^{i t}(\overline{u}^2-1).
\end{equation*}
But we expect constant $L$ to be much higher in Theorem \ref{thm:male2}.

\section{Definitions}

\subsection{Basic notions}
Let $(X,f)$ be a dynamical system on a compact metric space such that $f$ is a homeomorphism. By
\emph{(full) orbit} of $x$ we mean the set
\begin{equation*}
\orb (x,f)=\set{\dots, f^{-2} (x), f^{-1}(x), x, f(x), f^2 (x), \dots}.
\end{equation*}

A point $y\in X$ is an \emph{$\omega$-limit
point} (\emph{$\alpha$-limit point}) of a point $x$ if it is an accumulation point of the sequence
$x,f(x),f^2(x),\dots$ (resp. $x,f^{-1}(x),f^{-2}(x),\dots$). The set of all $\omega$-limit points ($\alpha$-limit points) of $x$
is called \emph{$\omega$-limit set} (resp. \emph{$\alpha$-limit set}) of $x$ and denoted
$\omega_f(x)$ (resp $\alpha_f(x)$). A point $p\in X$ is said to be \emph{periodic} if $f^n(p)=p$ for
some $n\geq 1$. The set of all periodic points for $f$ is denoted by
$\Per(f)$.

Let $(X,f)$, $(Y,g)$ be dynamical systems on compact metric spaces.
A continuous map $\Phi : X \ra Y$ is called a \emph{semiconjugacy} (or a \emph{factor map})
between $f$ and $g$ if $\Phi$ is surjective and $\Phi \circ f = g \circ \Phi$.

Let $Y$ be a topological space. For any set $Z\subset \R \times Y$ and $a,b,t \in \R$, $a<b$ we define
\begin{align*}
Z_t&=\{x\in Y: (t,x)\in Z\}, \\
Z_{[a,b]}&=\{(t,x)\in Z: t\in [a,b]\}.
\end{align*}

We denote by $\N$ the set of positive integers.

Let $c\in \C$ and $r>0$. Then $\overline{B}(c,r)\subset \C$ denotes the closed ball centered at~$c$ with radius $r$.

\subsection{Shift spaces}

Let $\A=\set{0,1,\ldots,n-1}$. We denote
$
\Sigma_n=\A^\Z .
$

By \emph{a word}, we mean any element of a free monoid $\A^*$ with the set of
generators equal to $\A$. If $x\in \Sigma_n$ and $i<j$ then by
$x_{[i,j]}$ we mean a sequence $x_i, x_{i+1}, \dots, x_j$. We may
naturally identify $x_{[i,j]}$ with the word $x_{[i,j]}=x_i
x_{i+1}\dots x_j\in \A^*$. It is also very convenient to denote $x_{[i,j)}=x_{[i,j-1]}$.

We introduce a metric $\rho$ in $\Sigma_n$ by
\begin{equation*}
\rho(x,y)=2^{-k}, \text{ where } k=\min\set{m\geq 0: x_{[-m,m]}\neq y_{[-m,m]}}.
\end{equation*}

If
$a_{-k}\ldots a_{-1}a_0a_1 \dots a_m\in \A^*$ then we define so called
\emph{cylinder set}:
\begin{eqnarray*}
{}[a_{-k}\dots a_m]&=&\set{x\in \Sigma_n \; : \; x_{[-k,m]}=a_{-k}\ldots a_{-1}a_0a_1 \dots a_m}.
\end{eqnarray*}
It is well known that cylinder sets form a neighborhood basis for
the space $\Sigma_n$.

By the $0^\infty$ we denote the element $x\in\Sigma_n$ such that $x_i=0$ for all $i\in \Z$. The usual map on $\Sigma_n$ is the shift map $\sigma$ defined by $\sigma(x)_i=x_{i+1}$ for all $i$. Dynamical system
$(\Sigma_n,\sigma)$ is called full two-sided shift over $n$ symbols.

\subsection{Dynamical systems and Wa\. zewski method}
Let $X$ be a topological space and $W$ be a subset of $X$. Denote by $\cl W$ the closure of $W$. The following definitions come from \cite{Srzed70}.

Let $D$ be an open subset of $\R\times X$. By a \emph{local flow} on $X$ we mean a continuous map $\phi:D\longrightarrow X$, such that three conditions are satisfied:
\begin{itemize}
\item[i)] $I_x=\{t\in \R:\; (t,x)\in D\}$ is an open interval $(\alpha_x,\omega_x)$ containing $0$, for every $ x\in X$,
\item[ii)] $ \phi(0,x)=x$, for every $x\in X$,
\item[iii)] $\phi(s+t,x)=\phi(t,\phi(s,x))$, for every $ x\in X$ and $s, t\in \R$ such that $s \in I_x$ and $ t\in I_{\phi(s,x)}$.
\end{itemize}
In the sequel we write $\phi_t (x)$ instead of $\phi (t,x)$.

Let $\phi$ be a local flow on $X$,  $x\in X$ and $W\subset X$. We call the set
\begin{align*}
\phi^+(x)=\phi([0,\omega_x)\times\{x\})\index{$\phi^+(x)$}
\end{align*}
the \emph{positive semitrajectory} of $x\in X$.

We distinguish three subsets of $W$ given by
\begin{align*}
W^-=&\{x\in W: \phi([0,t]\times\{x\})\not \subset W, \text{ for every }t>0\},\\
W^+=&\{x\in W: \phi([-t,0]\times\{x\})\not \subset W, \text{ for every }t>0\},\\
W^*=&\{x\in W: \phi(t,x)\not \in W, \text{ for some }t>0\}.
\end{align*}
It is easy to see that $W^-\subset W^*$. We call $W^-$ the \emph{exit set of $W$}, and $W^+$ the \emph{entrance set of $W$}.

We call $W$ a \emph{Wa\. zewski set} provided
\begin{enumerate}
\item if $x\in W$, $t>0$, and  $\phi([0,t]\times\{x\})\subset \cl W$ then  $\phi([0,t]\times\{x\})\subset W$,
\item $W^-$ is closed relative to $W^*$.
\end{enumerate}

\begin{prop}
\label{prop:wazewski}
If both $W$ and $W^-$ are closed subsets of $X$ then $W$ is a Wa\. zew\-ski set.
\end{prop}

The function $\sigma: W^*\longrightarrow [0,\infty)$
\begin{equation*}
\sigma(x)=\sup \{t\in [0,\infty): \phi([0,t]\times\{x\})\subset W\}
\end{equation*}
is called the \emph{escape-time function of $W$}.

The following lemma is called the Wa\. zewski lemma.
\begin{lem}[{\cite[Lemma 2.1 (iii)]{Srzed70}}]
\label{lem:wazewski}
Let $W$ be a Wa\. zewski set and $\sigma$ be its escape-time function. Then $\sigma$ is continuous.
\end{lem}

Finally, we state one version of the Wa\. zewski theorem.
\begin{thm}[{\cite[Corollary 2.3]{Srzed70}}]
\label{t-wazewskiego}
Let $\phi$ be a local flow on $X$, $W\subset X$ be a Wa\. zewski set and $Z\subset W$. If $W^-$ is not a strong deformation retract of $Z\cup W^-$ in $W$ then there exists an $x_0\in Z$ such that $\phi^+(x_0)\subset W$.
\end{thm}
\noindent (For the definition of the strong deformation retract see e.g. \cite{Dold}.)

\subsection{Processes}

Let $X$ be a topological space and $\Omega \subset \R\times X
\times \R$ be an open set.

By a \emph{local process} on $X$ we mean a continuous map
$\varphi:
\Omega \longrightarrow X$, such that three conditions
are satisfied: \begin{itemize}
\item[i)] $I_{(\sigma,x)}=\{t\in
\R:\; (\sigma,x,t)\in \Omega\}$ is an open interval containing $0$, for every $\sigma\in \R$ and $ x\in X$,
\item[ii)] $ \f(\sigma,\cdot,0)=\id_X$, for every $\sigma\in \R$,
\item[iii)] $\f(\sigma,x,s+t)=\f(\sigma+s,\f(\sigma,x,s),t)$, for every $ x\in X$, $\sigma\in \R$ and $s, t\in \R$ such that $s \in I_{(\sigma,x)}$ and $ t\in I_{(\sigma+s,\f(\sigma,x,s))}$. \end{itemize}
For abbreviation, we write $\f_{(\sigma,t)}(x)$ instead of
$\f(\sigma,x,t)$. If $\Omega = \R\times X \times \R$, then the process $\f$ is called \emph{global}.

Local process $\f$ on $X$ generates a local flow $\phi$ on $\R\times X$ by the formula
\begin{equation}
\label{eq:processflow}
\phi(t,(\sigma,x))=(\sigma+t,\f(\sigma, x,t)).
\end{equation}

Let $M$ be a smooth manifold and let $v: \R \times M
\longrightarrow TM$ be a time-dependent vector field. We assume
that $v$ is so regular that for every $(t_0,x_0)\in \R\times M$
the Cauchy problem
\begin{align}
\label{1row}
&\dot{x}= v(t,x),\\
\label{warp}
&x(t_0)= x_0
\end{align}
has unique solution. Then the equation \eqref{1row} generates a
local process $\f$ on $X$ by
$\f_{(t_0,t)}(x_0)=x(t_0,x_0,t+t_0)$, where $ x(t_0,x_0,\cdot)$ is the
solution of the Cauchy problem \eqref{1row}, \eqref{warp}.

Let $T$ be a positive number. In the sequel $T$ denotes the period. We assume that $v$ is $T$-periodic
in $t$. It follows that the local process $\f$ is $T$-periodic,
i.e.,
\begin{align*}
\f_{(\sigma+T,t)}=\f_{(\sigma,t)} \text{ for all } \sigma, t\in \R,
\end{align*}
hence there is a one-to-one correspondence between $T$-periodic
solutions of \eqref{1row} and fixed points of the Poincar\'{e} map
$\f_{(0,T)}$.

\subsection{Distributional chaos}
Let $\mathbb{N}$ denote the set of positive integers and let $f$ be a
continuous self map of a compact metric space $(X,\rho)$. We define
a function $\xi_f:X\times X \times \mathbb{R} \times \mathbb{N} \longrightarrow \mathbb{N}$
by:
\begin{equation*}
\xi_f (x,y,t,n) = \# \set{i\;:\; \rho(f^i(x),f^i(y))<t \;,\; 0\leq i
<n}
\end{equation*}
where $\# A$ denotes the cardinality of the set $A$. By the means of $\xi_f$
we define the following two functions:
\begin{equation*}
F_{xy}(f,t)=\liminf_{n \rightarrow \infty} \frac{1}{n} \xi_f
(x,y,t,n),\quad \quad F^*_{xy}(f,t)=\limsup_{n \rightarrow \infty}
\frac{1}{n} \xi_f (x,y,t,n).
\end{equation*}
For brevity, we often write $\xi$, $F_{xy}(t)$, $F^*_{xy}(t)$ instead of $\xi_f$, $F_{xy}(f,t)$, $F^*_{xy}(f,t)$ respectively.

Both functions $F_{xy}$ and $F^*_{xy}$ are nondecreasing,
$F_{xy}(t)=F^*_{xy}(t)=0$ for $t<0$ and $F_{xy}(t)=F^*_{xy}(t)=1$
for $t>\diam X$. Functions $F_{xy}$ and $F^*_{xy}$ are called
\emph{lower} and \emph{upper distribution} functions, respectively.

\begin{dfn}\label{dcdef}
A pair of points $(x,y)\in X\times X$ is called \emph{distributionally
chaotic (of type 1)} if
\begin{enumerate}
\item $F_{xy}(s)=0$ for some $s>0$, \label{cond:dc1a}
\item $F^*_{xy}(t)=1$ for all $t>0$ \label{cond:dc1b}.
\end{enumerate}

A set containing at least two points is called \emph{distributionally scrambled set of type 1} (or \emph{$d$-scrambled
set} for short) if any pair of its distinct points is distributionally
chaotic.

A map $f$ is \emph{distributionally chaotic (DC1)} if it has an
uncountable $d$-scramb\-led set. Distributional chaos is said to be
uniform if a constant $s$ from condition (\ref{cond:dc1a}) may be
chosen the same for all the pairs of distinct points of
$d$-scrambled set.
\end{dfn}

We remark here that the definition of
distributional chaos was introduced to extend approach proposed
by Li and Yorke in their famous paper \cite{LiYorke}. Then it is
clear why we use the name $d$-scrambled set. Namely each
$d$-scrambled set is also scrambled set as defined by Li and Yorke.

We also should mention that our notation is a slightly different compared to that introduced
by Schweizer and Sm\'{\i}tal
(founders of distributional chaos) in \cite{SchwSm}. It is mainly because
the definition of distributional chaos passed a very long journey since its introduction (even its name
changed as it was originally called strong chaos). The definition we present is one of the strongest
possibilities \cite{BaSmiSte} and is usually called distributional chaos of type $1$ to be distinguished from other two
weaker definitions - DC2 and DC3.

\begin{dfn}
We say that a $T$-periodic local process $\varphi$ on $M$ is \emph{(uniform)
distributionally chaotic} if there exists compact set $\Lambda
\subset M$ invariant for the Poincar\'{e} map $P_T=\varphi_{(0,T)}$
such that $P_T|_\Lambda$ is (uniform) distributionally chaotic.

We say that the equation \eqref{1row} is \emph{(uniform)
distributionally chaotic} if it generates a local process which is
(uniform) distributionally chaotic.
\end{dfn}

\subsection{Useful facts}

We recall Theorem 5 from \cite{OpWil3} which is crucial in the proof of Theorem \ref{thm:glowne} (it is also possible to use more general \cite[Theorem 11]{OpWil4}).

\begin{thm}\label{DC:thm2}
Let $f$ be a continuous map from a compact metric space
$(\Lambda,\rho)$ into itself and let $\Phi: \Lambda \ra \Sigma_2$ be
a semiconjugacy between $f$ and $\sigma$.

If there exists $x\in \Sigma_2 \cap \Per (\sigma)$ such that $\#
\Phi^{-1}(\set{x})=1$ then $ f $ is distributionally chaotic and
distributional chaos is uniform.
\end{thm}

\section{Main calculations}

\begin{proof}[Idea of the proof of Theorem \ref{thm:male2}]
There are many symmetries in the equation \eqref{egh:male2} e.g. function $z: (a,b)\rightarrow \C$ is a solution of the equation \eqref{thm:male2} if and only if the function $w$ such that $w(t)=e^{i\frac{2}{3}\pi}z\left(t+\frac{2}{3}\pi\right)$ is also a solution. Immediatelly one gets the third solution $x(t)=e^{i\frac{4}{3}\pi}z\left(t+\frac{4}{3}\pi\right)$.

The equation \eqref{thm:male2} with $\varepsilon=0$ has three stationary points which are saddle points $z_j=e^{-i\frac{1+4j}{6}\pi}$ for $j\in\set{1,2,3}$. So it is enough to investigate what happens between solutions $z_1$ and $z_2$.

It is important to notice that the line $l: \im(z)=-\frac{1}{2}$ is invariant. There is a heteroclinic solution between $z_1$ and $z_2$ contained in this line.

Term $-\varepsilon e^{it}$ makes stationary points continue to simple periodic solutions, but only for small $0<\varepsilon$. Obviously $\varepsilon<\frac{1}{2}$. So $\varepsilon_0 < \frac{1}{2}$.

The simple periodic solution $\xi_1$ continued from $z_1$ lies at both sides of the line $l$. The same happens for the simple periodic solution $\xi_2$ continued from $z_2$. But there are such time points $t$ that the $\xi_1(t)$ is above the line $l$, while $\xi_2(t)$ is below it. So starting near $\xi_1$ and above $l$, due to $L$ big enough, the solution enters neighbourhood of $\xi_2$ being above $l$. But on the other time $t$ the point $\xi_1(t)$ is below $l$ while $\xi_2(t)$ is above it. So starting near $\xi_1$ and below $l$, due to $L$ big enough, the solution enters neighbourhood of $\xi_2$ being below $l$. This gives intersection of unstable manifold of solution $\xi_1$ with the stable one of $\xi_2$.

To obtain a semiconjudacy to a sofic shift one needs to code the behaviour of solutions. Roughly speaking, being close to $\xi_1$ for a one period or transferring from the neighbourhood of $\xi_1$ to the neighbourhood of $\xi_2$ is coded by one letter $0$. Similarly, being close to $\xi_2$ for a one period or transferring from the neighbourhood of $\xi_2$ to the neighbourhood of $\xi_3$ is coded by one letter $1$. Of course, periodic solution $\xi_j$ is coded by the sequence $(j-1)^\infty.(j-1)^\infty$ for $j\in\set{1,2,3}$. And every sequence $(j-1)^\infty.(j-1)^\infty$ for $j\in\set{1,2,3}$ codes exactly one solution (which is $\xi_j$).

It seems that it is possible to get complicated dynamics for $L=\frac{1}{2}$ but the time of transfer from neighbourhood of $\xi_1$ to neighbourhood of $\xi_2$ must be longer than one period. This gives semiconjugacy to shift with lower entropy than $\left(\hat{\Sigma}, \sigma\right)$.
\end{proof}

\begin{proof}[Sketch of the proof of Theorem \ref{thm:glowne}]

Let us fix $R$ satisfying \eqref{ineq:R}. We define $a$, $\beta$, $\gamma$, $\delta$, $M$ to be such that
\begin{align}
\label{eq:a}
a=&0.7,\\
\label{eq:beta}
\beta = & \gamma =\delta=0.01.
\end{align}
are satisfied. 
By \eqref{ineq:lipschitz}, the equation \eqref{egh} generates a local process on $Q$. We denote it by $\f$ (we do need $\f$ to be a locall process on the whole $\C$ because we analyse dynamics only close to the origin). Observe that $\f$ is $2\pi$-periodic and if $f\equiv 0$, then, by Lemma \ref{lem:symmetry}, satisfies 
\begin{equation}
\label{eq:symmetryf}
-\f_{(\tau,t)}(z)=\f_{(\tau+\pi,t)}(-z) \text{ for every } \tau \in \R, z\in Q \text{ and } t\in I_{(\tau,z)}.
\end{equation}
Of course, if $f\not \equiv 0$, then \eqref{eq:symmetryf} is no longer valid, but since the perturbation $f$ is small the equality \eqref{eq:symmetryf} is still crucial for our calculations.

Write
\begin{align}
\label{eq:U}
U=&\set{(t,q)\in \R\times \C: \left|\re\left[(q-1)e^{-i\frac{t}{2}}\right]\right|\leq a, \left|\im\left[(q-1)e^{-i\frac{t}{2}}\right]\right|\leq a},\\
\label{eq:W}
W=&\set{(t,q)\in \R\times \C: \left|\re\left[(q+1)e^{-i\frac{t}{2}}\right]\right|\leq a, \left|\im\left[(q+1)e^{-i\frac{t}{2}}\right]\right|\leq a},\\
\label{eq:Z}
Z=&\{(t,q)\in \R\times \C: -0,5\leq \re[q]\leq 0.5, |\im[q]| \leq l\left((-1)^k\re[q]\right),\\
\nonumber
&t\in \left[k\pi-\beta,k\pi+\gamma \right] \text{ for some } k\in \Z\},
\end{align}
where
\begin{equation*}
l(x) = 
\begin{cases}
\frac{5}{28} - \frac{9}{14}x, &\text{ for } -0.5\leq x \leq 0.2,\\
0.05, &\text{ for } 0.2\leq x \leq 1.
\end{cases}
\end{equation*}
Moreover we write
\begin{align}
\label{eq:lambda}
\Lambda =& \{q\in \C: I_{(-\beta,q)}=\R \text{ and for every } k\in \Z \text{ there holds exactly one } \\
\nonumber
& \text{ of the conditions } \eqref{ineq:00}, \eqref{ineq:11}, \eqref{ineq:01},  \eqref{ineq:10}\}
\end{align}
where

\begin{align}
\label{ineq:00}
&\f_{(-\beta,t)}(q)\in U \text{ for every } t\in [k\pi, (k+1)\pi],\\
\label{ineq:11}
&\f_{(-\beta,t)}(q)\in W \text{ for every } t\in [k\pi, (k+1)\pi],\\
\label{ineq:01}
&\begin{cases}
\f_{(-\beta,k\pi)}(q)\in U,& \\
\f_{(-\beta,t)}(q)\in Z& \text{ for every } t\in [k\pi, k\pi+\beta+\gamma],\\
\f_{(-\beta,t)}(q)\in W& \text{ for every } t\in [k\pi+\beta+\gamma, (k+1)\pi],
\end{cases}\\
\label{ineq:10}
&\begin{cases}
\f_{(-\beta,k\pi)}(q)\in W,&\\
\f_{(-\beta,t)}(q)\in Z& \text{ for every } t\in [k\pi, k\pi+\beta+\gamma],\\
\f_{(-\beta,t)}(q)\in U& \text{ for every } t\in [k\pi+\beta+\gamma, (k+1)\pi].
\end{cases}
\end{align}

By Lemma \ref{lem:compactlambda}, the set $\Lambda$ is compact. Moreover it is invariant with respect to the Poincar\'{e} map $\f_{(-\beta,2\pi)}$, i.e. $\f_{(-\beta,2\pi)}(\Lambda)=\Lambda$.

Now we define a map $\Phi: \Lambda \ra \Sigma_2$ by
\begin{align}
\label{eq:semiconjugacy}
[\Phi(q)]_k=
\begin{cases}
0, & \text{ if } \eqref{ineq:00} \text{ or } \eqref{ineq:10} \text{ hold, }\\
1, & \text{ if } \eqref{ineq:11} \text{ or } \eqref{ineq:01} \text{ hold. }
\end{cases}
\end{align}
By the definition of $\Phi$, we get immediately
\begin{equation*}
\Phi\circ \f_{(-\beta,2\pi)}= \s^2 \circ \Phi.
\end{equation*}
Moreover, by the continuous dependence of solutions of \eqref{egh} on initial conditions, $\Phi$ is continuous.

Let $x\in \Sigma \subset \Sigma_2$ be such that there exists $N\in \N$ such that $x_j=0$ for every $|j|\geq N$, i.e. $x$ is homoclinic to $0^\infty$. By Lemma \ref{lem:homoorbita}, there exists $q_x\in \Lambda$ such that $\Phi(q_x)=x$. Since the set of homoclinic points to $0^\infty$ is dense in $\Sigma$, $\Phi$ is continuous and defined on a compact set $\Lambda$, we get
\begin{equation*}
\Phi(\Lambda)=\Sigma,
\end{equation*}
i.e. surjectivity of $\Phi$. Finally, $\Phi$ is a semiconjugacy between $\Lambda$ and $\Sigma$.

By Lemma \ref{lem:jedendojeden},
\begin{equation}
\label{eq:moc}
\# \Phi^{-1}(\{0^\infty\})=1
\end{equation}
holds so the existence of uniform distributional chaos follows by Theorem \ref{DC:thm2}.

By the symmetry \eqref{eq:symmetryf}, we get $\# \Phi^{-1}(\{1^\infty\})=1$. The existence of infinitely many heteroclinic solutions between $-1$ and $1$ follows by Lemma \ref{lem:hetero}. Namely, for every $x\in \Sigma$ such that there exists $N\in \N$ such that $x_j=0$, $x_{-j}=1$ or $x_j=1$, $x_{-j}=0$ for $j>N$ every point from $\Phi^{-1}(\{x\})$ is heteroclinic from $1$ to $-1$ or from $-1$ to $1$ respectively. 
\end{proof}

Now we present lemmas which were used in the proof of Theorem \ref{thm:glowne}. We use the notation introduced in the proof of Theorem \ref{thm:glowne} unless stated otherwise.

\begin{lem}
\label{lem:symmetry}
If $f\equiv 0$ then the condition \eqref{eq:symmetryf} holds.
\end{lem}

\begin{proof}
Let $\Delta\subset \R$ be an interval and $z:\Delta\ra \C$ be a solution of \eqref{egh}. It is enough to show that $\xi$ given by $\xi(t)=-z(t+\pi)$ is also a solution of \eqref{egh}. To see this let us calculate
\begin{align*}
\dot{\xi}(t)=-\dot{z}(t+\pi)=-Re^{i(t+\pi)}\left[\z^2(t+\pi)-1\right] = Re^{it}\left[\overline{\xi}^2(t)-1\right].
\end{align*}
\end{proof}

\begin{lem}
\label{lem:compactlambda}
The set $\Lambda$ given by \eqref{eq:lambda} is compact.
\end{lem}
\begin{proof}
It is enough to prove that $\Lambda$ is a closed subset of $\C$. Let $\set{z_n}_{n\in \N}\subset \Lambda$ be such that $\lim_{n \to \infty} z_n$ exists. We denote this limit by $z$.

Let us fix $t\in \R$. Then $\f_{(-\beta,t)}(z_n)\in \overline{B}(0,2)$ holds for every $n\in \N$, so $\lim_{n\to \infty} \f_{(-\beta,t)}(z_n)$ exists and it must be equal to $\f_{(-\beta,t)}(z)$, so $\f_{(-\beta,t)}(z)$ exists. By the arbitrariness of the choice of $t$, we get $I_{(-\beta,z)}=\R$.

Now we fix $k\in \Z$.

For every $z\in \Lambda$ let us write $\eta_z:\R\ra \C$ where $\eta_z(t)=\f(-\beta,z,t)$. By the continuous dependence on initial conditions,
\begin{equation*}
\lim_{n\to \infty} \eta_{z_n} = \eta_z \text{ uniformly on the interval } [k\pi, (k+1)\pi].
\end{equation*}
Thus exactly one condition among \eqref{ineq:00}, \eqref{ineq:11}, \eqref{ineq:01}, \eqref{ineq:10} is satisfied by almost all $\eta_{z_n}$. Since for every $t\in \R$ sets $U_t$, $W_t$ and $Z_t$ are compact (or empty), $\eta_{z}$ satisfies this condition.

Finally, $z\in \Lambda$.
\end{proof}

\begin{lem}
\label{lem:homoorbita}
Let $x\in \Sigma$ be such that there exists $N\in \N$ such that $x_j=0$ for every $|j|\geq N$. Then there exists $q_x\in \Lambda$ such that $\Phi(q_x)=x$.
\end{lem}

\begin{proof} First of all we assume that $f\equiv 0$.

In the sequel we investigate \eqref{egh} (especially in a neighbourhood of $1$) in the coordinates
\begin{align}
\label{eq:zmianaB}
w=w(q,t)=(q-1)e^{-i\frac{t}{2}}
\end{align}
which has the form
\begin{equation}
\label{eq:B}
\dot{w}=u(t,w)=2R\overline{w}+ Re^{-i\frac{t}{2}}\overline{w}^2 -\frac{i}{2}w + e^{-i\frac{t}{2}} f\left( t, we^{i\frac{t}{2}}+1 \right).
\end{equation}
We also investigate \eqref{egh} (especially in a neighbourhood of $-1$) in the coordinates
\begin{equation}
\label{eq:zmianaC}
p=p(q,t)=(q+1)e^{-i\frac{t}{2}}
\end{equation}
which has the form
\begin{equation}
\label{eq:C}
\dot{p}=\tilde{u}(t,p)=-2R\overline{p}+ Re^{-i\frac{t}{2}}\overline{p}^2 -\frac{i}{2}p + e^{-i\frac{t}{2}} f\left( t, pe^{i\frac{t}{2}}-1 \right).
\end{equation}
We denote by $\psi$ and $\tilde{\psi}$ the locall processes generated by \eqref{eq:B} and \eqref{eq:C} respectively. Let us stress that
\begin{align*}
q=&\Upsilon(w,t)=e^{i\frac{t}{2}}w+1,\\
q=& \Xi(p,t)= e^{i\frac{t}{2}}p-1
\end{align*}
hold. Thus the following equalities
\begin{align*}
\f(\tau,q,t)=&\Upsilon(\psi(\tau,w(q,\tau),t),\tau +t),\\
\psi(\tau,w,t)=& w(\f(\tau,\Upsilon(w,\tau),t),\tau+t),\\
\f(\tau,q,t)=&\Xi(\tilde{\psi}(\tau,p(q,\tau),t),\tau +t),\\
\tilde{\psi}(\tau,p,t)=& p(\f(\tau,\Xi(p,\tau),t),\tau+t)
\end{align*}
are satisfied wherever they have sense.

It is easy to see, that
\begin{align*}
w(U)=&\set{w\in  \C: |\re[w]|\leq a, |\im[w]|\leq a},\\
p(W)=&\set{p\in  \C: |\re[p]|\leq a, |\im[p]|\leq a}.
\end{align*}
Let us notice that inside $w(U)$ the vector field $u$ is close to $2R\overline{w}$. The other terms are treated as perturbation so the qualitative behaviour of $u$ inside $\R\times w(U)$ is just as the term $2R\overline{w}$.
Similarly, the qualitative behaviour of $\tilde{u}$ inside $\R \times p(W)$ is just as the term $-2R\overline{p}$.

Let
\begin{align}
\label{eq:defK}
K=&\set{w\in \C: |\re[w]|\leq \frac{11}{10}\beta, |\im[w]|\leq 2\beta^2},\\
\nonumber
L=&\set{p\in \C: |\re[p]|\leq 2\beta^2, |\im[p]|\leq \frac{11}{10}\beta}.
\end{align}
By Lemma \ref{lem:niestabilna}, there exists a continuous function
\begin{equation*}
\xi: \left[-\frac{11}{10}\beta,\frac{11}{10}\beta\right]\ni o \mapsto \xi(o)\in \left[-2\beta^2,2\beta^2\right]
\end{equation*}
such that for every $o\in \left[-\frac{11}{10}\beta,\frac{11}{10}\beta\right]$ we get
\begin{align*}
\lim_{t\to -\infty} \psi(\delta,o+i\xi(o),t)=&0,\\
\psi(\delta,o+i\xi(o),t)\in &K \text{ for every } t\leq 0
\end{align*}
where $\delta=-\beta-2(N-1)\pi$.

It immediatelly follows by Lemma \ref{lem:przeskoki}, that there exists an interval $[\mu,\nu]\subset \left[-\frac{11}{10}\beta,\frac{11}{10}\beta\right]$ such that the following conditions hold: \eqref{ineq:przedzialmunu}, exactly one out of \eqref{eq:bezobrotu} and \eqref{eq:zobrotem}, for every $l\in \set{1,2,\dots, 2N}$ exactly one out of \eqref{ineq:pom00}, \eqref{ineq:pom11}, \eqref{ineq:pom01} and \eqref{ineq:pom10} where

\begin{align}
\label{ineq:przedzialmunu}
&\psi(\delta,o+i\xi(o),4N\pi) \in K \text{ for every } o \in [\mu, \nu],\\
\label{eq:bezobrotu}
&
\re[\psi(\delta,\mu+i\xi(\mu),4N\pi)]=-\frac{11}{10}\beta, \; \re[\psi(\delta,\nu+i\xi(\nu),4N\pi)]=\frac{11}{10}\beta,
\\
\label{eq:zobrotem}
&\re[\psi(\delta,\mu+i\xi(\mu),4N\pi)]=\frac{11}{10}\beta, \;
\re[\psi(\delta,\nu+i\xi(\nu),4N\pi)]=-\frac{11}{10}\beta,
\\
\label{ineq:pom00}
&
\begin{cases}
\text{ if } x_{-N+l-1}=x_{-N+l}=0, \text{ then for every } o \in [\mu, \nu]\\
\psi(\delta,o+i\xi(o),t) \in K \text{ for every } t\in [(l-1)\pi, l\pi],
\end{cases}\\
\label{ineq:pom11}
&
\begin{cases}
\text{ if } x_{-N+l-1}=x_{-N+l}=1, \text{ then for every } o \in [\mu, \nu]\\
\tilde{\psi}(\delta,p(\Upsilon(o+i\xi(o),\delta),\delta),t) \in L \text{ for every } t\in [(l-1)\pi, l\pi],
\end{cases}
\\
\label{ineq:pom01}
&
\begin{cases}
\text{ if } x_{-N+l-1}=0, x_{-N+l}=1, \text{ then for every } o \in [\mu, \nu]\\
\psi(\delta,o+i\xi(o),(l-1)\pi) \in K,\\
\f(\delta,\Upsilon(o+i\xi(o),\delta),t)\in Z \text{ for every } t\in [(l-1)\pi, (l-1)\pi+\beta+\gamma],\\
\tilde{\psi}(\delta,p(\Upsilon(o+i\xi(o),\delta),\delta),t) \in p(W) \text{ for every } t\in [(l-1)\pi+\beta+\gamma, l\pi],\\
\tilde{\psi}(\delta,p(\Upsilon(o+i\xi(o),\delta),\delta),l\pi) \in L,
\end{cases}
\\
\label{ineq:pom10}
&
\begin{cases}
 \text{ if } x_{-N+l-1}=1, x_{-N+l}=0, \text{ then for every } o \in [\mu, \nu]\\
\tilde{\psi}(\delta,p(\Upsilon(o+i\xi(o),\delta),\delta),(l-1)\pi) \in L, \\
\f(\delta,\Upsilon(o+i\xi(o),\delta),t)\in Z \text{ for every } t\in [(l-1)\pi, (l-1)\pi+\beta+\gamma],\\
 \psi(\delta,o+i\xi(o),t) \in w(U)  \text{ for every } t\in [(l-1)\pi+\beta+\gamma, l\pi],\\
\psi(\delta,o+i\xi(o),l\pi) \in K.
\end{cases}
\end{align}

Reversing time in \eqref{eq:B} and applying Lemma \ref{lem:niestabilna}, we get the existence of a continuous function
\begin{equation*}
\tilde{\xi}: \left[-\frac{11}{10}\beta,\frac{11}{10}\beta\right]\ni o \mapsto \tilde{\xi}(o)\in [-2\beta^2,2\beta^2]
\end{equation*}
such that for every $o\in \left[-\frac{11}{10}\beta,\frac{11}{10}\beta\right]$ we get
\begin{align*}
\lim_{t\to +\infty} \psi(-\beta + 2(N+1)\pi,\tilde{\xi}(o)+io,t)=&0,\\
\psi(-\beta + 2(N+1)\pi,\tilde{\xi}(o)+io,t)\in &w(U) \text{ for every } t\geq 0.
\end{align*}

By the continuity of $\psi$ and \eqref{ineq:przedzialmunu}, \eqref{eq:bezobrotu}, \eqref{eq:zobrotem}, there exist $\widehat{o}\in [\mu,\nu]$, $\widehat{y}\in \left[-\frac{11}{10}\beta,\frac{11}{10}\beta\right]$ such that
\begin{equation*}
\psi(\delta,\widehat{o}+i\xi(\widehat{o}),4N\pi)=\tilde{\xi}(\widehat{y})+i\widehat{y}
\end{equation*}
holds.

Write $q_x= \Upsilon(\psi(\delta,\widehat{o}+i\xi(\widehat{o}),2(N-1)\pi),-\beta)$. It is easy to see that $\Phi(q_x)=x$ holds.

The proof for the case $f\not \equiv 0$ is similar. The role of trivial solutions $-1$ and $1$ plays periodic ones which are contuined from them.
\end{proof}

In the following lemma we use the notation introduced in the proofs of Theorem \ref{thm:glowne} and Lemma \ref{lem:homoorbita}.

\begin{lem}
\label{lem:niestabilna}
For every $\tau\in \R$ there exists a continuous function
\begin{equation*}
\xi: \left[-\frac{11}{10}\beta,\frac{11}{10}\beta\right]\ni o \mapsto \xi(o)\in [-2\beta^2,2\beta^2]
\end{equation*}
such that for every $o\in \left[-\frac{11}{10}\beta,\frac{11}{10}\beta\right]$ we get
\begin{align*}
\lim_{t\to -\infty} \psi(\tau,o+i\xi(o),t)=&0,\\
\psi(\tau,o+i\xi(o),t)\in &K \text{ for every } t\leq 0.
\end{align*}
\end{lem}
\begin{proof}
To use a Wa\. zewski method, let us reverse the time in \eqref{eq:B} by setting $a(t)=w(-t)$. We get
\begin{equation}
\label{eq:niestabilna1}
\dot{a}=\widehat{u}(t,a)=-2R\overline{a} - Re^{i\frac{t}{2}}\overline{a}^2 + \frac{1}{2}a.
\end{equation}
Let $\widehat{\psi}$ and $\widehat{\phi}$ denote the local process on $\C\setminus \set{0}$ and the local flow on $\R\times (\C \setminus \set{0})$, respectively, generated by \eqref{eq:niestabilna1}. Of course, the relation between $\widehat{\psi}$ and $\widehat{\phi}$ is given by \eqref{eq:processflow}.

Let us fix $\tau\in \R$. To finish the proof it is enough to show that there exists a continuous function
\begin{equation*}
\xi: \left[-\frac{11}{10}\beta,\frac{11}{10}\beta\right]\ni o \mapsto \xi(o)\in [-2\beta^2,2\beta^2]
\end{equation*}
such that $\xi(0)=0$ and for every $o\in \left[-\frac{11}{10}\beta,\frac{11}{10}\beta\right]\setminus \set{0}$ we get
\begin{align}
\label{eq:niestabilna2}
\lim_{t\to +\infty} \widehat{\psi}(\tau,o+i\xi(o),t)=&0,\\
\label{ineq:niestabilna1}
\widehat{\psi}(\tau,o+i\xi(o),t)\in &K \text{ for every } t\geq 0.
\end{align}

Let us fix $o\in \left(0,\frac{11}{10}\beta\right]$. We define
\begin{align*}
\Gamma = \set{(t,a)\in \R\times \C: t\in \R, \re[a]\in \left(0,\frac{11}{10}\beta\right], |\Arg[a]| \leq \beta}.
\end{align*}
We show that
\begin{align}
\label{eq:Gamma-}
\Gamma^- =  \set{(t,a)\in \R\times \C: t\in \R, \re[a]\in \left(0,\frac{11}{10}\beta\right], |\Arg[a]| = \beta}.
\end{align}
We parameterize part of $\partial \Gamma$ by
\begin{align*}
s_1:\R\times\left(0, \frac{11\beta}{10\cos(\beta)}  \right)\ni (t,\theta) \mapsto (t, \theta e^{i\beta}).
\end{align*}
An outward orthonormal vector to this part of $\partial \Gamma$ is given by
\begin{align*}
n_1:\R\times\left(0, \frac{11\beta}{10\cos(\beta)}  \right)\ni (t,\theta) \mapsto \left(0, i e^{i\beta}\right).
\end{align*}
The inner product of the outward orhonormal vector and the vector field $\widehat{u}$ has the form
\begin{align*}
\langle n_1(t,\theta), \widehat{u}(s_1(t,\theta))\rangle =& \re\Big[-ie^{-i\beta}(-2)R\theta e^{-i\beta} -ie^{-i\beta}(-R) e^{i\frac{t}{2}} \theta^2 e^{-2i\beta} \\
& -ie^{-i\beta} \theta \frac{1}{2} e^{i\beta} \Big]\\
\geq & 2R \sin(2\beta) - R \theta^2 -\frac{1}{2}\theta\\
>& 0.
\end{align*}
Another part of $\partial \Gamma$ can be parameterized by
\begin{align*}
s_2:\R\times\left(0, \frac{11\beta}{10\cos(\beta)}  \right)\ni (t,\theta) \mapsto (t, \theta e^{-i\beta}).
\end{align*}
An outward orthonormal vector to this part of $\partial \Gamma$ is given by
\begin{align*}
n_2:\R\times\left(0, \frac{11\beta}{10\cos(\beta)}  \right)\ni (t,\theta) \mapsto \left(0, -i e^{-i\beta}\right).
\end{align*}
By calculations similar to the above, one can see that $\langle n_1(t,\theta), \widehat{u}(s_1(t,\theta))\rangle >0$ holds for every $(t,\theta) \in \left(0, \frac{11\beta}{10\cos(\beta)}  \right)$.

Now let $(t,a)\in \Gamma$ i.e. $t\in \R$, $\re[a]\in \left(0, \frac{11}{10}\beta\right]$ and \begin{equation}
\label{ineq:zbiegado0}
|\im[a]|\leq \tan({\beta})\re[a]
\end{equation}
hold. We calculate
\begin{align}
\nonumber
\re[\widehat{u}(t,a)]=& \re\left[-2R\overline{a} - Re^{i\frac{t}{2}}\overline{a}^2 + \frac{1}{2}a \right]\\
\label{ineq:Gammalewo}
\leq & -2R \re[a] +R|a|^2 +\frac{1}{2}|a|\\
\nonumber
\leq & \re[a] \left(-2R + R \re[a] (1+\tan^2(\beta))+ \frac{1}{2}\sqrt{1+\tan^2(\beta)}\right)\\
\nonumber
<& 0.
\end{align}
Finally, the vector field $\widehat{u}$ points outwards $\Gamma$ on some part of $\partial \Gamma$ and points inwards on the other part of $\partial \Gamma$. Thus \eqref{eq:Gamma-} holds.

Let us notice, that since both $\Gamma$ and $\Gamma^-$ are closed in $\R\times (\C\setminus\set{0})$, by Proposition \ref{prop:wazewski}, $\Gamma$ is a Wa\. zewski set.

Write
\begin{equation*}
\Theta =\set{(t,a)\in \Gamma: t= \tau, \re[a]=o}.
\end{equation*}
Since $\Gamma^-$, as a not connected set, is not a strong deformation retract of a connected $\Gamma^-\cup \Theta$ in $\Gamma$, by Theorem \ref{t-wazewskiego}, there exists $a_0\in \Theta_\tau$ such that
\begin{equation}
\label{ineq:Gammaphiplus}
{\widehat{\phi}}^+(\tau,a_0)\subset \Gamma
\end{equation}
holds.

We set $\xi(o)=\im[a_0]$ (if there are more than one $a_0$'s, we choose one of them). To define $\xi$ for negative $o$'s we repeat the above construction with $\widehat{\Gamma}$ instead of $\Gamma$ where $\widehat{\Gamma} = \set{(t,a)\in \R\times \C: (t,-a) \in \Gamma }$. The same ralation holds between  $\widehat{\Gamma}^-$ and $\Gamma^-$. All calculations are similar to the above due to the symmetries of $-2R\overline{a}$ which is the leading term of $\widehat{u}$ close to the origin.

Let us notice, that by \eqref{ineq:Gammaphiplus}, we immediately get \eqref{ineq:niestabilna1}. Moreover, by \eqref{ineq:Gammalewo}, we get \eqref{eq:niestabilna2}.

To finish the proof it is enough to show that $\xi$ is continuous. By \eqref{ineq:zbiegado0}, $\xi$ is continuous at $0$.

To obtain a contradiction, let us assume that there exists $o_1 \in \left[-\frac{11}{10}\beta,\frac{11}{10}\beta\right]\setminus \set{0}$ such that
\begin{equation}
\label{ineq:contracontinous}
\xi(o_1)\neq \lim_{o\ra o_1} \xi(o)
\end{equation}
holds.

Let $\chi$ be a solution of \eqref{eq:niestabilna1} satisfying $\chi(\tau) = o_1+i\xi(o_1)$. Let us notice that
\begin{equation}
\label{ineq:ogrchi}
|\chi(t)|\leq \frac{12}{10}\beta
\end{equation}
for every $t\geq \tau$.

We make the change of variables $\zeta = a - \chi$ and get
\begin{equation}
\label{eq:naupsilon}
\dot{\zeta}= \mathring{u}(t,\zeta) = -2R \overline{\zeta} - Re^{i\frac{t}{2}}\overline{\zeta}(\overline{\zeta} +2 \overline{\chi})+ \frac{1}{2}\zeta.
\end{equation}

We define
\begin{align*}
\widetilde{K}=& \set{\zeta\in \C: |\im[\zeta]|\leq \frac{11}{10}\beta, |\im[\zeta]|\geq |\re[\zeta]|},\\
\widehat{K}= & \set{\zeta\in \widetilde{K}: |\im[\zeta]|= \frac{11}{10}\beta}.
\end{align*}
We show that every solution $\varsigma$ of \eqref{eq:naupsilon} such that $\varsigma(\tau) \in \widetilde{K}$ leaves $\widetilde{K}$ throught $\widehat{K}$ and it happens for some $t\geq \tau$. Keeping in mind \eqref{ineq:ogrchi}, for any $\zeta\in \widetilde{K}$ let us estimate
\begin{align*}
|\im[\dot{\zeta}]| \geq & |\im[-2R \overline{\zeta}]| - \left|\im\left[ Re^{i\frac{t}{2}}\overline{\zeta}(\overline{\zeta}+ 2 \overline{\chi})\right] \right| - \frac{1}{2}|\im[\zeta]|\\
\geq & |\im[\zeta]|\left( 2R -\frac{1}{2}\right)  - R |\zeta|(|\zeta|+ 2|\chi|)\\
\geq & |\im[\zeta]|\left( 2R -\frac{1}{2} -R \sqrt{2}\frac{11\sqrt{2}+24}{10}\beta\right)\\
\geq & R |\im[\zeta]|.
\end{align*}
Let us investigate the behaviour of $\mathring{u}$ on $\partial \widetilde{K}$. We parameterize a part of $\partial \widetilde{K}$ by
\begin{equation*}
s_3:\left(0, \frac{11}{10}\beta\right] \ni \theta \mapsto \theta (1+i).
\end{equation*}
An outward orthogonal vector is given by $n_3= 1-i$. The inner product of the outward orthogonal vector and the vector field $\mathring{u}$ has the form
\begin{align*}
\langle n_3(\theta), \mathring{u}(t,s_3(\theta))\rangle =& \re\Big[(1+i)(-2)R\theta (1-i) \\
&- (1+i)(-R) e^{i\frac{t}{2}} \theta(1-i)(\theta(1-i) + 2\overline{\chi}(t))  +(1+i)^2 \theta \frac{1}{2}  \Big]\\
\leq & -4R\theta + 2 R \theta \left(\sqrt{2}\theta +\frac{24}{10}\beta\right) +\theta\\
<& 0.
\end{align*}
So on this side of boundary of $\widetilde{K}$ the vector field points inwards $\widetilde{K}$. By the symmetries of $\widetilde{K}$ and the dominating term $-2R\overline{\zeta}$ of $\mathring{u}$ close to the origin, the vector field $\mathring{u}$ points inwards $\widetilde{K}$ on the other sides of the boundary, except $\widehat{K}$. Finally, every nonzero $\varsigma$ leaves $\widetilde{K}$ for some time $t\geq \tau$ throught $\widehat{K}$.

So if \eqref{ineq:contracontinous} holds, then there exists a solution $\widetilde{\varsigma} $ of  \eqref{eq:niestabilna1} satisfying $\widetilde{\varsigma}(\tau) = o +i\xi(o)$ for some $o$ close to $o_1$ such that $\varsigma_1(\tau)= \widetilde{\varsigma}(\tau) -\chi(\tau) \in \widetilde{K}$. But, as we know,  $\varsigma_1$ must leave $\widetilde{K}$ for some time $t\geq \tau$ throught $\widehat{K}$. It means that $\widetilde{\varsigma}$ leaves $K$ for some time $t\geq \tau$ which contradicts \eqref{ineq:Gammaphiplus}. 
\end{proof}

In the following lemma we use the notation introduced in the proofs of Theorem \ref{thm:glowne} and Lemma \ref{lem:homoorbita}.

\begin{lem}
 \label{lem:przeskoki}
There exists an interval $[\mu,\nu]\subset \left[-\frac{11}{10}\beta,\frac{11}{10}\beta\right]$ such that the following conditions hold: \eqref{ineq:przedzialmunu}, exactly one out of \eqref{eq:bezobrotu} and \eqref{eq:zobrotem}, for every $l\in \set{1,2,\dots, 2N}$ exactly one out of \eqref{ineq:pom00}, \eqref{ineq:pom11}, \eqref{ineq:pom01} and \eqref{ineq:pom10}.
\end{lem}
\begin{proof}
Let
\begin{align*}
K_1=&\set{w\in K: \re w =-\frac{11}{10}\beta},\\
K_2=&\set{w\in K: \re w =\frac{11}{10}\beta},\\
K_3=&\set{w\in K: \im w =-2\beta^2},\\
K_4=&\set{w\in K: \im w =2\beta^2},\\
L_1=&\set{p\in L: \im p =-\frac{11}{10}\beta},\\
L_2=&\set{p\in L: \im p =\frac{11}{10}\beta}.
\end{align*}
We prove more than \eqref{ineq:pom00}, \eqref{ineq:pom11}, \eqref{ineq:pom01} and \eqref{ineq:pom10}. Namely, we prove that there exist sequences $\set{\mu_j}_{j=0}^{2N}$, $\set{\nu_j}_{j=0}^{2N}$ such that $\mu_0=-\frac{11}{10}\beta$, $\nu_0=\frac{11}{10}\beta$, $\mu_j< \mu_{j+1}$, $\nu_j> \nu_{j+1}$ for $j\in \set{0,1,\ldots,2N-1}$, $\mu_{2N}< \nu_{2N}$ and
\begin{align*}
&
\begin{cases}
\text{in the case \eqref{ineq:pom00} }\\
\psi(\delta,o+i\xi(o),t) \in K \text{ for every } t\in [(l-1)\pi, l\pi], o \in [\mu_l, \nu_l]\\
\text{and either } \psi(\delta,\mu_l+i\xi(\mu_l),l\pi) \in K_1, \psi(\delta,\nu_l+i\xi(\nu_l),l\pi) \in K_2\\
\text{or } \psi(\delta,\mu_l+i\xi(\mu_l),l\pi) \in K_2, \psi(\delta,\nu_l+i\xi(\nu_l),l\pi) \in K_1,
\end{cases}\\
&
\begin{cases}
\text{in the case \eqref{ineq:pom11} }\\
\tilde{\psi}(\delta,p(\Upsilon(o+i\xi(o),\delta),\delta),t) \in L \text{ for every } t\in [(l-1)\pi, l\pi], o \in [\mu_l, \nu_l]\\
\text{and either }\\
\tilde{\psi}(\delta,p(\Upsilon(\mu_l+i\xi(\mu_l),\delta),\delta),l\pi) \in L_1, \tilde{\psi}(\delta,p(\Upsilon(\nu_l+i\xi(\nu_l),\delta),\delta),l\pi) \in L_2\\
\text{or }\\
\tilde{\psi}(\delta,p(\Upsilon(\mu_l+i\xi(\mu_l),\delta),\delta),l\pi) \in L_2, \tilde{\psi}(\delta,p(\Upsilon(\nu_l+i\xi(\nu_l),\delta),\delta),l\pi) \in L_1,
\end{cases}\\
&
\begin{cases}
\text{in the case \eqref{ineq:pom01} } \text{ for every } o \in [\mu_l, \nu_l]\\
\psi(\delta,o+i\xi(o),(l-1)\pi) \in K,\\
\f(\delta,\Upsilon(o+i\xi(o),\delta),t)\in Z \text{ for every } t\in [(l-1)\pi, (l-1)\pi+\beta+\gamma],\\
\tilde{\psi}(\delta,p(\Upsilon(o+i\xi(o),\delta),\delta),t) \in p(W) \text{ for every } t\in [(l-1)\pi+\beta+\gamma, l\pi],\\
\tilde{\psi}(\delta,p(\Upsilon(o+i\xi(o),\delta),\delta),l\pi) \in L,\\
\text{and either }\\
\tilde{\psi}(\delta,p(\Upsilon(\mu_l+i\xi(\mu_l),\delta),\delta),l\pi) \in L_1, \tilde{\psi}(\delta,p(\Upsilon(\nu_l+i\xi(\nu_l),\delta),\delta),l\pi) \in L_2\\
\text{or }\\
\tilde{\psi}(\delta,p(\Upsilon(\mu_l+i\xi(\mu_l),\delta),\delta),l\pi) \in L_2, \tilde{\psi}(\delta,p(\Upsilon(\nu_l+i\xi(\nu_l),\delta),\delta),l\pi) \in L_1,
\end{cases}\\
\end{align*}

\begin{align*}
\begin{cases}
\text{in the case \eqref{ineq:pom10} } \text{ for every } o \in [\mu_l, \nu_l]\\
\tilde{\psi}(\delta,p(\Upsilon(o+i\xi(o),\delta),\delta),(l-1)\pi) \in L, \\
\f(\delta,\Upsilon(o+i\xi(o),\delta),t)\in Z \text{ for every } t\in [(l-1)\pi, (l-1)\pi+\beta+\gamma],\\
\psi(\delta,o+i\xi(o),t) \in w(U)  \text{ for every } t\in [(l-1)\pi+\beta+\gamma, l\pi],\\
\psi(\delta,o+i\xi(o),l\pi) \in K\\
\text{and either } \psi(\delta,\mu_l+i\xi(\mu_l),l\pi) \in K_1, \psi(\delta, \nu_l+i\xi(\nu_l), l\pi) \in K_2\\
\text{or } \psi(\delta,\mu_l+i\xi(\mu_l),l\pi) \in K_2, \psi(\delta,\nu_l+i\xi(\nu_l),l\pi) \in K_1
\end{cases}
\end{align*}
hold.

The case \eqref{ineq:pom00} comes directly from Lemma \ref{lem:pelnyopisU}, because for every $t\in \R$ the vector field $u(\cdot,t)$ from \eqref{eq:B} points outward $K$ on $K_1\cup K_2$ and points inward $K$ on $K_3\cup K_4$. Indeed, to see this let us fix $w\in K_1$. Then
\begin{align}
\nonumber
\re[w] =& \re\left[2R\overline{w}+ Re^{-i\frac{t}{2}}\overline{w}^2 -\frac{i}{2}w  \right]\\
\nonumber
\leq & -2R\frac{11}{10}\beta + R \left|e^{-i\frac{t}{2}}\overline{w}^2\right| +
\left|\frac{i}{2}w \right|\\
\label{ineq:pomK1}
\leq & -R\left(\frac{11}{5}\beta - \left(\frac{11}{10}\beta\right)^2 - 4\beta^4\right) + \frac{11}{10}\beta + 2 \beta^2\\
\nonumber
<&0.
\end{align}
Similar calculations show that $\re[w]>0$ for every $w\in K_2$. Let us fix now $w\in K_3$. Then
\begin{align}
\nonumber
\im[w]=& \im \left[2R\overline{w}+ Re^{-i\frac{t}{2}}\overline{w}^2 -\frac{i}{2}w  \right]\\
\nonumber
\geq & 4R\beta^2 -R \left|e^{-i\frac{t}{2}}\overline{w}^2\right| - \frac{1}{2}\re[w] \\
\label{ineq:pomK3}
\geq& R \left(4\beta^2 - \frac{121}{100}\beta^2 - 4\beta^4\right) - \frac{11}{5}\beta\\
\nonumber
\geq& R \left(\frac{279}{100}\beta^2 - 4\beta^4 - \frac{22}{5}\beta^3\right)\\
\nonumber
>&0.
\end{align}
Similar calculations show that $\im[w]<0$ for every $w\in K_4$.

The case \eqref{ineq:pom11} is very similar to the case \eqref{ineq:pom00} (we only need to interchange the horizontal and vertical directions). By the $2\pi$-periodicity of $\f$ and symmetry \eqref{eq:symmetryf}, the case \eqref{ineq:pom10} is equivalent to \eqref{ineq:pom01}, so it is enough to prove the statement in the case \eqref{ineq:pom01}.

By the $2\pi$-periodicity of $\f$, there is no loss of generality in assuming that $\delta=-\beta$ and $l=1$.

So at the beginnig $\psi(\delta,o+i\xi(o),(l-1)\pi) = \psi(-\beta,o+i\xi(o),0) \in K$ holds for every $o\in \left[\mu_{l-1},\nu_{l-1}\right]=\left[\mu_{0},\nu_{0}\right]$. Thus $\f(-\beta,\Upsilon(o+i\xi(o),-\beta),0)\in Z$ for every $o\in \left[\mu_{0},\nu_{0}\right]$.

Let
\begin{align}
\label{eq:dzeta}
\tilde{\zeta}=&\beta\frac{\left(1-e^{-4R\beta} \right)(2-\beta)}{2-\beta+e^{-4R\beta}},\\
\nonumber
K_{\tilde{\zeta}} =& \set{w \in K: -\beta\leq \re [w]\leq -\beta+\tilde{\zeta}},\\
\nonumber
\widetilde{K}_{\tilde{\zeta}} =& \set{w \in K: \re [w]= -\beta+\tilde{\zeta}},\\
\nonumber
\widetilde{K}_{\beta} =& \set{w \in K: \re [w]= -\beta},\\
\nonumber
\tilde{L}=&\set{p\in \C: |\re[p]|\leq \frac{11}{10}\beta, |\im[p]|\leq 3\beta},\\
\nonumber
\tilde{L}_1=&\set{p\in \tilde{L}: \im[p] = -3\beta},\\
\nonumber
\tilde{L}_2=&\set{p\in \tilde{L}: \im[p] = 3\beta},\\
\nonumber
\tilde{L}^u=&\set{p\in \tilde{L}: \im[p]>4\beta^2},\\
\nonumber
\tilde{L}^l=&\set{p\in \tilde{L}: \im[p]<-4\beta^2}.
\end{align}
Obviously, $\tilde{\zeta}<\beta$.

We finish the proof in the following steps:
\begin{enumerate}
\item \label{en:pierwszy}
showing that for every $w\in K_{\tilde{\zeta}}$ there exists $t_w\in 
\left(
0, \beta+ \gamma\right]$ such that \begin{equation}
\label{ineq:pom20}
p\left(\f(-\beta, \Upsilon(w,-\beta),t_w),-\beta + t_w\right) \in \tilde{L}
\end{equation}
 and
\begin{equation}
\label{ineq:pom21}
\f(-\beta, \Upsilon(w,-\beta),t)\in Z_{-\beta + t} \text{ for every } t\in \left[0,t_w\right]
\end{equation}
hold,
\item \label{en:drugi}
showing that for every $w\in K_\beta$ the inclusion
\begin{equation*}
p\left(\f(-\beta, \Upsilon(w,-\beta),\gamma-\beta),-2\beta + \gamma\right) \in \tilde{L}^u
\end{equation*}
holds,
\item \label{en:trzeci}
showing that for every $w\in \widetilde{K}_{\tilde{\zeta}}$ the inclusion
\begin{equation*}
p\left(\f(-\beta, \Upsilon(w,-\beta),\gamma+\beta), \gamma\right) \in \tilde{L}^l
\end{equation*}
holds,
\item \label{en:czwarty}
showing that for every $\sigma\in \R$ and $p\in \tilde{L}^u$ there exists $t\in[0,\beta]$ such that $\tilde{\psi}(\sigma,p,t) \in \tilde{L}_2$,
\item \label{en:piaty}
showing that for every $\sigma\in \R$ and $p\in \tilde{L}^l$ there exists $t\in[0,\beta]$ such that $\tilde{\psi}(\sigma,p,t) \in \tilde{L}_1$,
\item \label{en:szosty}
showing that for every $\sigma\in \R$ and $z\in p\left(W_\sigma\times \set{\sigma}\right)$ there exists $t\in[0,\beta]$ such that $\tilde{\psi}(\sigma,p,t) \in L$ or $|\im[\tilde{\psi}(\sigma,p,t)]|=\frac{11}{10}\beta$,
\item \label{en:siodmy}
observing that, by above steps and Lemma \ref{lem:pelnyopisU}, for any curve contained in $K$ which connects $K_1$ and $K_2$ there must exists its connected part which is after transfer via flow contained in $L$ and connects $L_1$ and $L_2$.
\end{enumerate}

To follow the steps let us observe that, by \eqref{ineq:RN} and \eqref{ineq:N}, one gets the inequalities
\begin{align*}
\re[v(t,z)] \leq & N + R\left[ - \cos \beta +\frac{(\re z)^2}{\cos \beta}\right] \\
< & 0
\end{align*}
for every $(t,z)\in [-\beta, \beta]\times \set{z\in \C: |\re[z]|\leq 0.98, |\im[z]|\leq 1}$. It means that the vector field points to the left in the whole set. This information combined with Lemma \ref{lem:pomrow1} gives some estimates how quickly one can move from set $U$ to set $W$ through set $Z$.

\end{proof}

\begin{lem}
\label{lem:pelnyopisU}
Let $\alpha, \beta, \gamma, \delta, \zeta, \eta, a,b \in \R$, $\alpha < \beta$, $\gamma<\delta$, $\zeta < \eta$, $a<b$, $K=[\alpha, \beta]\times [\gamma,\delta]\subset \R^2$ and the time dependent vector field $v: \R\times \R^2\ni (t,x,y)\mapsto v(t,x,y)=(v_1(t,x,y),v_2(t,x,y))\in \R^2$ be continuous and so regular that the equation
\begin{equation*}
\begin{cases}
\dot{x}=&v_1(t,x,y),\\
\dot{y}=&v_2(t,x,y)
\end{cases}
\end{equation*}
generates a local process $\f$ on $\R^2$.
Write
\begin{align*}
K_1=&\set{(x,y)\in K:x=\alpha},\\
K_2=&\set{(x,y)\in K:x=\beta},\\
K_3=&\set{(x,y)\in K:y=\gamma},\\
K_4=&\set{(x,y)\in K:y=\delta}.
\end{align*}
Let the conditions
\begin{align}
\label{en:naK1}
v_1(t,x,y)<0 & \text{ for every } t\in \R, (x,y)\in K_1,\\
\label{en:naK2}
v_1(t,x,y)>0 & \text{ for every } t\in \R, (x,y)\in K_2,\\
\label{en:naK3}
v_2(t,x,y)>0 & \text{ for every } t\in \R, (x,y)\in K_3,\\
\label{en:naK4}
v_2(t,x,y)<0 & \text{ for every } t\in \R, (x,y)\in K_4
\end{align}
hold. Let $\xi\in \c([\zeta, \eta], K)$ be such that
\begin{align}
\label{ineq:wK1}
\xi(\zeta)\in& K_1,\\
\label{ineq:wK2}
\xi(\eta)\in& K_2
\end{align}
hold. Then there exist $\mu,\nu\in \R$, $\zeta<\mu<\nu<\eta$ such that
\begin{align*}
\f_{(a,t)}(\xi(p))\in &K \text{ for every } t\in [0,b-a], p\in [\mu,\nu],\\
\f_{(a,b-a)}(\xi(\mu))\in& K_1,\\
\f_{(a,b-a)}(\xi(\nu))\in& K_2
\end{align*}
hold.
\end{lem}

\begin{proof}
Let $\lambda>0$. Write $K^\lambda=[\alpha-\lambda,\beta+\lambda]\times [\gamma,\delta]\subset \R^2$ and
\begin{align*}
K_1^\lambda =& \set{(x,y)\in K^\lambda: x=\alpha-\lambda},\\
K_2^\lambda =& \set{(x,y)\in K^\lambda: x=\beta+\lambda},\\
K_3^\lambda =& \set{(x,y)\in K^\lambda: y=\gamma},\\
K_4^\lambda =& \set{(x,y)\in K^\lambda: y=\delta}.
\end{align*}
By the continuity of $v$ and compactness of $[a,b]\times K$, there exists $\lambda>0$ such that the qualitative behaviour of $v$ on the $\partial K^\lambda$ is the same as on the $\partial K$ i.e. the vector field points inwards on $K_3^\lambda$ and $K_4^\lambda$ and points outwards on $K_1^\lambda$ and $K_2^\lambda$ (the inequalities analogous to \eqref{en:naK1}--\eqref{en:naK4} hold).

By the continuity of $\f$, compactness of $K$ and the fact that the vector field points on $K^\lambda_3$ and $K^\lambda_4$ inward $K^\lambda$, there exists $\rho>0$ such that for every $\tau\in [a,b]$ the condition
\begin{equation}
\label{ineq:fimale}
\f_{(\tau,t)}(K)\subset K^{\lambda}
\end{equation}
holds for every $t\in [0,\rho]$.

Let us fix $\tau\in [a,b)$. Since the interval $[a,b]$ can be divided into finitely many intervals which lenghts are not greater then $\rho$, it is enough to prove that there exist $\mu,\nu\in \R$, $\zeta<\mu<\nu<\eta$ such that
\begin{align}
\label{ineq:k1}
\f_{(\tau,t)}(\xi(p))\in &K \text{ for every } t\in [0,\rho], p\in [\mu,\nu],\\
\label{ineq:k2}
\f_{(\tau,\rho)}(\xi(\mu))\in& K_1,\\
\label{ineq:k3}
\f_{(\tau,\rho)}(\xi(\nu))\in& K_2
\end{align}
hold.

Write
\begin{align*}
J_1=&\set{(x,y)\in \R^2: x\in [\alpha-\lambda,\alpha]},\\
J_2=&\set{(x,y)\in \R^2: x\in [\beta, \beta+\lambda]}.
\end{align*}

Since, by \eqref{ineq:fimale}, \eqref{ineq:wK1} and \eqref{ineq:wK2}, $\f_{(\tau,\rho)}(\xi(\zeta))\in J_1$ and $\f_{(\tau,\rho)}(\xi(\eta))\in J_2$, there are points $p_1, p_2 \in (\zeta, \eta)$ such that $\f_{(\tau,\rho)}(\xi(p_1))\in K_1$ and $\f_{(\tau,\rho)}(\xi(p_2))\in K_2$.

We set
\begin{align}
\mu=&\sup\set{p\in[\zeta, \eta]: \f_{(\tau,\rho)}(\xi(p))\in K_1},\\
\nu=&\inf\set{p\in[\mu, \eta]: \f_{(\tau,\rho)}(\xi(p))\in K_2}.
\end{align}
We claim that conditions \eqref{ineq:k1}--\eqref{ineq:k3} hold.

Indeed, \eqref{ineq:k2} and  \eqref{ineq:k3} hold by the continuity of $\f$, $\xi$, compactness of $K_1$ and $K_2$, respectively.

To obtain a contradiction, we assume that \eqref{ineq:k1} does not hold. Then there exists $p\in (\mu, \nu)$ such that either $\f_{(\tau,t)}(\xi(p))\in J_1\setminus K$ or $\f_{(\tau,t)}(\xi(p))\in J_2\setminus K$ for some $t\in (0,\rho)$.

Without losing of generality we may assume that $\f_{(\tau,t)}(\xi(p))\in J_1\setminus K$ holds. Since the part of trajectory of $\xi(p)$ cannot enter $K$ and must stay in $K^\e$ for times from interval $[t,\rho]$, it must stay in $J_1$. So $\f_{(\tau,\rho)}(\xi(p))\in J_1\setminus K_1$. The first coordinate of $\f_{(\tau,\rho)}(\xi(p))$ is lower than $\alpha$ and the first coordinate of $\f_{(\tau,\rho)}(\nu)$ is equal to $\beta$, so there exists $q\in (p,\nu)$ such that the first coordinate of $\f_{(\tau,\rho)}(q)$ is equal to $\alpha$. But it means that $\f_{(\tau,\rho)}(q)\in K_1$. Since $q>\mu$ we obtain the desired contadiction.
\end{proof}

\begin{lem}
\label{lem:pomrow1}
Let \eqref{ineq:R}, \eqref{ineq:RN} and \eqref{eq:beta} be satisfied. The local flow $\phi$ on $\R$ generated by the equation
\begin{equation}
\label{eq:syspom1}
\dot{x}= N - R\cos \beta + \frac{R}{\cos \beta} x^2
\end{equation}
is given by
\begin{align}
\nonumber
\phi(t,x,y)= & \sqrt{\cos^2 \beta - \frac{N}{R}\cos \beta} \\
\label{eq:sysrozw1}
& \cdot \left(\frac{2}{
1 - \frac{x\sqrt{R}- \sqrt{R\cos^2 \beta - N \cos \beta}}{x\sqrt{R}+ \sqrt{R\cos^2 \beta - N \cos \beta}} \exp\left( 2 t \sqrt{R^2 - \frac{NR}{\cos \beta}}  \right)
}-1 \right).
\end{align}
\end{lem}
\begin{proof}
The proof is a matter of straightforward computation and is left to the reader.
\end{proof}

\begin{lem}
\label{lem:jedendojeden}
The equation \eqref{eq:moc} holds.
\end{lem}

\begin{proof}
We use the notation from Lemma \ref{lem:homoorbita}.

Let us assume that $f\equiv 0$. 

Since the vector field $u$ has a dominating term $2R\overline{w}$ inside the set $\R\times w(U)$ the qualitative behaviour of $u$ is the same as $\R\times w(U)$ (saddle node). So the only solution staying in $w(U)$ for all times is the trivial one.

When $f\equiv 0$ is no longer valid then the trivial solution of \eqref{eq:B} continues to the periodic one $\kappa$. By \eqref{ineq:RN} and \eqref{ineq:N}, it can be shown that $|\re[\kappa]| < 0.0051$ and $|\im[\kappa]| < 0.0051$.

Now by the change of variables $y=q-\kappa(t)$ the equation \eqref{eq:B} has the form 
\begin{align}
\label{eq:D}
\dot{y} = & m(t,y) = 2R\overline{y} + Re^{-i\frac{t}{2}}\overline{y}^2 + 2Re^{-i\frac{t}{2}}\overline{y}\overline{\kappa} -\frac{i}{2}y \\ 
\nonumber 
&+ e^{-i\frac{t}{2}}\left[ f(t,(y+\kappa)e^{i\frac{t}{2}}+1)-f(t,\kappa e^{i\frac{t}{2}}+1)
\right].
\end{align}

The dominating term of the vector field $m$ is $2R\overline{y}$ so, by \eqref{ineq:lipschitz}, situation is qualitatively the same as in the case of $f\equiv 0$.


\end{proof}

\begin{lem}
\label{lem:hetero}
Let $(X,d)$, $(Y,\rho)$ be compact metric spaces, $f\in \c(X)$, $g\in \c(Y)$ be homeomorphisms and $\Phi \in \c(X,Y)$ be a semiconjugacy between $f$ and $g$. Let $y_1, y_2 \in \Per(g)$ be such that
\begin{equation}
\label{eq:hetero}
\Phi^{-1}(\{y_i\})=\set{o_i}
\end{equation}
holds for $i\in \set{1,2}$ where $\set{o_1, o_2}\subset \Per(f)$. Let $y\in Y$ be such that $\alpha_g(y)=\orb(y_1,g)$ and $\omega_g(y)=\orb(y_2,g)$ hold. Then every point $o\in \Phi^{-1}(\{y\})$ satisfies $\alpha_f(o)=\orb(o_1,f)$ and $\omega_f(o)=\orb(o_2,f)$.
\end{lem}

\begin{proof}
I. Let us fio $y\in Y$ such that $\alpha_g(y)=\orb(y_1,g)$ holds and $o\in \Phi^{-1}(\{y\})$. We show that $\alpha_f(o)=\orb(o_1,f)$.

Let $n\in \N$ be period of $y_1$. It easy to see, that by \eqref{eq:hetero}, point $o_1$ is also $n$-periodic and
\begin{equation}
\label{eq:hetero1}
\Phi^{-1}(\orb(y_1,g))=\orb(o_1,f)
\end{equation}
holds.

To obtain a contradiction, let us assume that there eoists $p\in \alpha_f(o)$ such that $p\not \in \orb(o_1,f)$. It means that there eoists sequence $\set{k_j}_{j\in \N}\subset \Z\setminus \N$ such that $\lim_{j\to \infty} f^{k_j}(o)=p$ holds. But
\begin{align*}
\Phi(p)=&\Phi\left(\lim_{j\to \infty} f^{k_j}(o)\right)= \lim_{j\to \infty} \Phi\left(f^{k_j}(o)\right)= \lim_{j\to \infty} (g^{k_j}\circ\Phi)(o) \\
=& \lim_{j\to \infty} g^{k_j}(y) \in \alpha_g(y)=\orb(y_1,g)
\end{align*}
which contradict \eqref{eq:hetero1}. Finally, $\alpha_f(o) \subset \orb(o_1,f)$, which immediately gives $\alpha_f(o) = \orb(o_1,f)$.

II. Let us now fio $y\in Y$ such that $\omega_g(y)=\orb(y_2,g)$ holds and $o\in \Phi^{-1}(\{y\})$. To finish the proof it is enough to show that $\omega_f(o)=\orb(o_1,f)$. The proof is similar to the one from part I.
\end{proof}

\section{Further remarks}

Symmetry \eqref{eq:symmetryf} is not essential in our investigations. It only simplifies calculations.

\section{Acknowledgements}
This paper is supported by the Faculty of Mathematics and Information Science, Warsaw University of Technology grant No. 504/02482/1120 for 2016 year.

\bibliographystyle{plain}
\def\cprime{$'$} \def\cprime{$'$}
  \def\soft#1{\leavevmode\setbox0=\hbox{h}\dimen7=\ht0\advance \dimen7
  by-1ex\relax\if t#1\relax\rlap{\raise.6\dimen7
  \hbox{\kern.3ex\char'47}}#1\relax\else\if T#1\relax
  \rlap{\raise.5\dimen7\hbox{\kern1.3ex\char'47}}#1\relax \else\if
  d#1\relax\rlap{\raise.5\dimen7\hbox{\kern.9ex \char'47}}#1\relax\else\if
  D#1\relax\rlap{\raise.5\dimen7 \hbox{\kern1.4ex\char'47}}#1\relax\else\if
  l#1\relax \rlap{\raise.5\dimen7\hbox{\kern.4ex\char'47}}#1\relax \else\if
  L#1\relax\rlap{\raise.5\dimen7\hbox{\kern.7ex
  \char'47}}#1\relax\else\message{accent \string\soft \space #1 not
  defined!}#1\relax\fi\fi\fi\fi\fi\fi} \def\cprime{$'$}
  \def\polhk#1{\setbox0=\hbox{#1}{\ooalign{\hidewidth
  \lower1.5ex\hbox{`}\hidewidth\crcr\unhbox0}}}
  \def\polhk#1{\setbox0=\hbox{#1}{\ooalign{\hidewidth
  \lower1.5ex\hbox{`}\hidewidth\crcr\unhbox0}}}
  \def\polhk#1{\setbox0=\hbox{#1}{\ooalign{\hidewidth
  \lower1.5ex\hbox{`}\hidewidth\crcr\unhbox0}}}
  \def\soft#1{\leavevmode\setbox0=\hbox{h}\dimen7=\ht0\advance \dimen7
  by-1ex\relax\if t#1\relax\rlap{\raise.6\dimen7
  \hbox{\kern.3ex\char'47}}#1\relax\else\if T#1\relax
  \rlap{\raise.5\dimen7\hbox{\kern1.3ex\char'47}}#1\relax \else\if
  d#1\relax\rlap{\raise.5\dimen7\hbox{\kern.9ex \char'47}}#1\relax\else\if
  D#1\relax\rlap{\raise.5\dimen7 \hbox{\kern1.4ex\char'47}}#1\relax\else\if
  l#1\relax \rlap{\raise.5\dimen7\hbox{\kern.4ex\char'47}}#1\relax \else\if
  L#1\relax\rlap{\raise.5\dimen7\hbox{\kern.7ex
  \char'47}}#1\relax\else\message{accent \string\soft \space #1 not
  defined!}#1\relax\fi\fi\fi\fi\fi\fi}
  \def\soft#1{\leavevmode\setbox0=\hbox{h}\dimen7=\ht0\advance \dimen7
  by-1ex\relax\if t#1\relax\rlap{\raise.6\dimen7
  \hbox{\kern.3ex\char'47}}#1\relax\else\if T#1\relax
  \rlap{\raise.5\dimen7\hbox{\kern1.3ex\char'47}}#1\relax \else\if
  d#1\relax\rlap{\raise.5\dimen7\hbox{\kern.9ex \char'47}}#1\relax\else\if
  D#1\relax\rlap{\raise.5\dimen7 \hbox{\kern1.4ex\char'47}}#1\relax\else\if
  l#1\relax \rlap{\raise.5\dimen7\hbox{\kern.4ex\char'47}}#1\relax \else\if
  L#1\relax\rlap{\raise.5\dimen7\hbox{\kern.7ex
  \char'47}}#1\relax\else\message{accent \string\soft \space #1 not
  defined!}#1\relax\fi\fi\fi\fi\fi\fi}
  \def\soft#1{\leavevmode\setbox0=\hbox{h}\dimen7=\ht0\advance \dimen7
  by-1ex\relax\if t#1\relax\rlap{\raise.6\dimen7
  \hbox{\kern.3ex\char'47}}#1\relax\else\if T#1\relax
  \rlap{\raise.5\dimen7\hbox{\kern1.3ex\char'47}}#1\relax \else\if
  d#1\relax\rlap{\raise.5\dimen7\hbox{\kern.9ex \char'47}}#1\relax\else\if
  D#1\relax\rlap{\raise.5\dimen7 \hbox{\kern1.4ex\char'47}}#1\relax\else\if
  l#1\relax \rlap{\raise.5\dimen7\hbox{\kern.4ex\char'47}}#1\relax \else\if
  L#1\relax\rlap{\raise.5\dimen7\hbox{\kern.7ex
  \char'47}}#1\relax\else\message{accent \string\soft \space #1 not
  defined!}#1\relax\fi\fi\fi\fi\fi\fi}

\end{document}